\documentclass[twoside,12pt]{article}
\usepackage{graphicx,amsmath,latexsym,amssymb,amsthm}
\usepackage[colorlinks=black,urlcolor=black]{hyperref}
at 9pt  

\date{}
\newtheorem{theorem}{Theorem}[section]

\numberwithin{equation}{section}

\begin{document}
\setlength{\unitlength}{1cm}



\vskip1.5cm

 \centerline { \textbf{Modified Expansion Theorem  for Sturm-Liouville problem }}
 \centerline { \textbf{with transmission conditions }}

\vskip.2cm


\vskip.8cm \centerline {\textbf{K.Aydemir$^*$ and O. Sh. Mukhtarov
$^*$ }}

\vskip.5cm

\centerline {$^*$Department of Mathematics, Faculty of Science,}
\centerline {Gaziosmanpa\c{s}a University,
 60250 Tokat, Turkey}
\centerline {e-mail : {\tt omukhtarov@yahoo.com,
kadriye.aydemir@gop.edu.tr }}



\vskip.5cm \hskip-.5cm{\small{\bf Abstract :}This paper is devoted
to the derivation of expansion a associated with a discontinuous
Sturm-Liouville problems defined on $[-\pi, 0)\cup(0,\pi]$. We
derive an eigenfunction expansion theorem for the Green's function
of the problem as well as a theorem of uniform convergence of a
certain class of functions.


\section{\textbf{Introduction}}
The importance of Sturm-Liouville problems for spectral methods lies
in the fact that the spectral approximation of the solution of a
differential equation is usually regarded as a finite expansion of
eigenfunctions of a suitable Sturm-Liouville problem. The issue of
expansion in generalized eigenfunctions is a classical one going
back at least to Fourier.  A relatively recent impact is due to the
study of wave propagation in random media \cite{ge, st}, where
eigenfunction expansions are an important input in the proof of
localization. The use of this tool is settled by classical results
in the Schrödinger operator case. But with the study of operators
related with classical waves, \cite{ge, st1}, a need for more
general results on eigenfunction expansion became apparent.
Eigenfunction expansions  problems for Sturm-Liouville problems have
been investigated by many authors, see \cite{le1, ti}

 In this paper we shall
investigate one discontinuous eigenvalue problem which consists of
Sturm-Liouville equation,
\begin{equation}\label{1.1}
\Gamma (y):=-y^{\prime \prime }(x,\lambda)+ q(x)y(x,\lambda)=\lambda
y(x,\lambda)
\end{equation}
to hold in finite interval $(-\pi, \pi) $ except at one inner point
$0 \in (-\pi, \pi) $ , where discontinuity in $u  \ \textrm{and} \
u'$ are prescribed by transmission conditions
\begin{equation}\label{1.2}
\Gamma_{1}(y):=a_{1}y'(0-,\lambda)+a_{2}y(0-,\lambda)+a_{3}y'(0+,\lambda)+a_{4}y(0+,\lambda)=0,
\end{equation}
\begin{equation}\label{1.3}
\Gamma_{2}(y):=b_{1}y'(0-,\lambda)+b_{2}y(0-,\lambda)+b_{3}y'(0+,\lambda)+b_{4}y(0+,\lambda)=0,
\end{equation}
together with the  boundary conditions
\begin{equation}\label{1.4}
 \Gamma_{3}(y):=\cos \alpha y(-\pi,\lambda)+\sin\alpha y'(-\pi,\lambda)=0,
\end{equation}
\begin{equation}\label{1.5}
\Gamma_{4}(y):=\cos\beta y(\pi,\lambda)+\sin\beta y'(\pi,\lambda)=0.
\end{equation}
 Such problems with point interactions are
also studied in \cite{os1}, etc. Boundary value problems which may
have discontinuities in the solution or its derivative at an
interior point  are also studied. Conditions are imposed on the left
and right limits of solutions and their derivatives at an interior
point
 and are often called transmission conditions" or interface
conditions". These problems often arise in varied assortment of
physical transfer problems, see \cite{li, ya}. Also, some problems
with transmission conditions arise in thermal conduction problems
for a thin laminated plate (i.e., a plate composed by materials with
different characteristics piled in the thickness).

 It is the purpose of this paper to present a new and somewhat extensive family of
orthogonal expansions associated with an discontinuous
Sturm-liouville problem. By using an own technique we introduce a
new equivalent inner product in the Hilbert space
$L_{2}(-\pi,0)\oplus L_{2}(0,\pi)$ and a linear operator in it such
a way that the considered problem can be interpreted as eigenvalue
problem for this operator. Using a new approach to the now classical
theory of Sturm-Liouville eigenfunction expansions, based
essentially on the method of integral equations. Moreover, we
introduce a expansion method for solution of carleman's equation.

\section{ Preliminary Results
 }
Let T=$ \left[%
\begin{array}{cccc}
  a_{1} & a_{2} & a_{3} & a_{4} \\
  b_{1} & b_{2} & b_{3} & b_{4}
  \\
\end{array} %
 \right]. $
 Denote  the determinant of the i-th
 and
j-th columns of the matrix T  by $\rho_{ij}$. Note that throughout
this study we shall assume that $ \rho_{12}>0 \ \ \textrm{and} \
\rho_{34}>0.$

In this section we shall define two basic solutions $\phi(x,\lambda)
\ \textrm{and} \ \chi(x,\lambda)$ by own technique as follows.  At
first, let us consider solutions of the equation (\ref{1.1})  on the
left-hand $\left[ -\pi,0\right)$ of the considered interval
$\left[\pi,0\right)\cup(0,\pi]$ satisfying the initial conditions
\begin{eqnarray}\label{(2.10)}
\label{tam1}
 \ \ y(-\pi,\lambda)=\sin\alpha , \
\ \frac{\partial y(-\pi,\lambda)}{\partial x}=-\cos\alpha
\end{eqnarray}
By virtue of well-known existence and uniqueness theorem of ordinary
differential equation theory  this initial-value problem for each
$\lambda $ has a unique solution $\phi _{1}(x,\lambda )$. Moreover
[\cite{ti}, Teorem 7] this solution
is an entire function of $%
\lambda $ for each fixed $x\in \left[ -\pi,0\right).$ Using this
solutions we can prove that the equation (\ref{1.1}) on  the
right-hand interval $\in (0,\pi]$ of the considered interval
$\left[-\pi,0\right)\cup(0,\pi]$ has the solution
$u=\phi_2(x,\lambda)$ satisfying the initial conditions
\begin{eqnarray}\label{8}
&&y(0,\lambda) =\frac{1}{\rho_{12}}(\rho_{23}\phi_{1}(0,\lambda
)+\rho_{24}\frac{\partial\phi_{1}(0,\lambda )}{\partial x})\\
&& \label{9} y^{\prime }(0,\lambda)
=\frac{-1}{\rho_{12}}(\rho_{13}\phi_{1}(0,\lambda
)+\rho_{14}\frac{\partial\phi_{1}(0,\lambda )}{\partial x}).
\end{eqnarray}
By applying the method of \cite{os1}  we can prove that  the
equation (\ref{1.1}) on $ (0,\pi]$ has an unique solution
$\phi_{2}(x,\lambda)$ satisfying the conditions (\ref{8})-(\ref{9})
which also is an entire function of the parameter $\lambda $ for
each fixed $x \in (0,\pi]$. Consequently, the function  $\phi
(x,\lambda )$ defined by
\begin{equation}
\phi (x,\lambda )=\{
\begin{array}{c}
\phi _{1}(x,\lambda )\text{ \ for }x\in \lbrack -\pi,0) \\
\phi _{2}(x,\lambda )\text{ \ for }x\in (0,\pi].%
\end{array}
\label{(3.16)}
\end{equation}
 satisfies equation $(\ref{1.1})$,  the first boundary condition
$(\ref{1.4})$ and the both transmission conditions $(\ref{1.2})$ and $(%
\ref{1.3})$. Similarly, $\chi_{2}(x,\lambda)$   be solutions of
equation (\ref{1.1}) on the left-right interval $(0, \pi]$ subject
to initial conditions
\begin{equation}\label{4}
 \ \ y(\pi,\lambda)=-\sin\beta , \
\  \frac{\partial y(\pi,\lambda)}{\partial x}=\cos\beta.
\end{equation}
By virtue of [\cite{ti}, Teorem 7] each of these solutions are
entire functions of $\lambda$ for fixed x. By applying the same
technique we can prove  there is an unique solution
$\chi_{1}(x,\lambda)$ of equation (\ref{1.1}) the left-hand interval
$[-\pi, 0) \ $ satisfying the initial condition
\begin{eqnarray}\label{11}
&&y(0,\lambda) =\frac{-1}{\rho_{34}}(\rho_{14}\chi_{2}(0,\lambda
)+\rho_{24}\frac{\partial\chi_{2}(0,\lambda )}{\partial x}),\\
&& \label{12} y^{\prime }(0,\lambda)
=\frac{1}{\rho_{34}}(\rho_{13}\chi_{2}+\Delta_{23}\frac{\partial\chi_{2}(0,\lambda
)}{\partial x}).
\end{eqnarray}
 By applying the similar technique as in \cite{os1} we
can prove that the solutions $\ \chi_{1}(x,\lambda)$ are also an
entire functions of parameter $\lambda$ for each fixed x.
Consequently, the function $ \chi(x,\lambda)$ defined  as
\begin{displaymath} \chi(x,\lambda)=\left
\{\begin{array}{ll}
\chi_{1}(x,\lambda), & x\in [-\pi, 0) \\
\chi_{2}(x,\lambda), & x\in (0, \pi] \\
\end{array}\right.
\end{displaymath}
satisfies the equation (\ref{1.1}) on whole $[-\pi, 0)\cup (0,
\pi]$, the other boundary condition (\ref{1.5}) and the both
transmission conditions (\ref{1.2}) and (\ref{1.3}).
 In the Hilbert Space
$\mathcal{H}=L_{2}[-\pi,0)  \oplus L_{2}(0,\pi] $ of two-component
vectors we define an inner product by
$$
<y,z>_{\mathcal{H}}:= \ \rho_{12} \int\limits_{-\pi}^{0}
y(x)\overline{z(x)}dx + \rho_{34} \int\limits_{0}^{\pi}
y(x)\overline{z(x)}dx
$$
for $y= y(x), \ z= z(x) \ \in \mathcal{H}$. We introduce the linear
operator $A:\mathcal{H}\rightarrow \mathcal{H}$  with domain of
definition satisfying the
following conditions\\
 i) $y \ \textrm{and} \ y'$ are absolutely
continuous in each of intervals $[-\pi,0)$ \ and $(0,\pi]$ and has a
finite limits $y(c\mp) \ \textrm{and} \ y'(c\mp)$\\
 ii) $\Gamma y(x) \in
\mathcal{H}, \ \Gamma_1 y(x)=\Gamma_2 y(x)=\Gamma_3 y(x)=\Gamma_4 y(x)=0,$\\
Obviously $D(A)$ is a linear subset dense in $\mathcal{H}$. We put
$$(Ay)(x)=\Gamma y(x), \ x\in \mathcal{H}
$$ for $y \in D(A)$. Then the problem $(\ref{1.1})-(\ref{1.5})$ is
equivalent to the equation
$$Ay=\lambda y $$ in the Hilbert space $\mathcal{H}$.
 Taking in view that the   Wronskians
  $ W(\phi_{i},\chi_{i};x) :=
\phi_{i}(x,\lambda)\chi'_{i}(x,\lambda)-\phi'_{i}(x,\lambda)\chi_{i}(x,\lambda)$
\ are independent of variable x  we shall denote  $
w_{i}(\lambda)=W(\phi_{i}, \chi_{i};x) \ (i=1,2)$. By using
(\ref{1.2}) and (\ref{1.3}) we have $ \label{7}
\rho_{12}w_{1}(\lambda)= \rho_{34}w_{2}(\lambda) $ for each $\lambda
\in \mathbb{C}$. It is convenient to introduce the notation
 \begin{equation}\label{kn}w(\lambda):=\rho_{34} w_{1}(\lambda) =  \rho_{12}
w_{2}(\lambda).\end{equation}
\begin{theorem}
For all $y, z \in D(A),$  the equality
\begin{eqnarray}\label{sim}<Ay,z>=<y,Az>
\end{eqnarray}
holds.
\end{theorem}
\begin{proof}
 Integrating by parts we have for all $y, z \in D(A),$
\begin{eqnarray}\label{2.z}
<Ay,z> &=& \rho_{12}\int\limits_{-\pi}^{0} \Gamma
y(x)\overline{z(x)}dx +\rho_{34}\int_{0}^{\pi} \Gamma
y(x)\overline{z(x)}dx
\nonumber\\
&=& \rho_{12}\int\limits_{-\pi}^{0} y(x)\overline{\Gamma z(x)}dx
+\rho_{34}\int\limits_{0}^{\pi} y(x)\overline{\Gamma z(x)}dx  +
\rho_{12}W[y, \overline{z};0-]\nonumber\\&-&\rho_{12}W[y,
\overline{z};-\pi] +\rho_{34} W[y,\overline{z};\pi] - \rho_{34} \
W[y,\overline{z};0]
\nonumber\\
&=&\ <y,Az>  + \rho_{12}W[y_{0}, \overline{z};0] - \rho_{12} W[y,
\overline{z};-\pi]\nonumber\\ &+&\rho_{34} W[y,\overline{z};\pi] -
\rho_{34}\ W[y,\overline{z};0]
\end{eqnarray}
From the boundary  conditions $(\ref{1.2})$-$(\ref{1.3})$ it is
follows obviously that
\begin{equation}\label{c1}W(y, \overline{z};-\pi)=0
 \ \textrm{and} \  W(y, \overline{z};\pi)=0
\end{equation}
The direct calculation gives
\begin{equation}\label{c3}\rho_{12}W(y,
\overline{z};0) = \rho_{34}W(y, \overline{z};0).
\end{equation}
Substituting (\ref{c1}) and (\ref{c3}) in (\ref{2.z}) we obtain the
equality (\ref{sim}).
\end{proof}
Relation (\ref{sim}) shows that the operator A is symmetric and
selfa-adjoint. Therefore all eigenvalues of the operator A are real
and two eigenfunctions corresponding to the distinct eigenvalues are
orthogonal in the sense of the following equality
\begin{eqnarray}\label{2.3}\rho_{12}\int\limits_{-\pi}^{0} y(x)z(x)dx + \rho_{34} \int\limits_{0}^{\pi}
y(x)z(x)dx=0.
\end{eqnarray}

\section{Expansion theorem by Green function's method%
} To present its explicit form we introduce the Green function

\begin{eqnarray}\label{2.15}
G(x,s;\lambda)=\left\{\begin{array}{c}
 \frac{\phi(x,\lambda)\chi(s,\lambda)}{\omega(\lambda)} \ \ \ \ \textrm{for} \  -\pi\leq s\leq x\leq \pi, \ \  x, s \neq 0\\
                \\
  \frac{\phi(s,\lambda)\chi(x,\lambda)}{\omega(\lambda)} \ \ \ \ \textrm{for} \  -\pi\leq x\leq s\leq \pi, \ \  x, s\neq 0 \\
             \end{array}\right.
\end{eqnarray}
for $x, s \in [-\pi,0)\cup(0,\pi]$ where $\phi(x,\lambda) \ \textrm{
and} \ \chi(x,\lambda)$ are solutions of the boundary-value-
transmission problem (\ref{1.1})-(\ref{1.5}). It symmetric with
respect to x and s, and real-valued for real $\lambda$. We show that
the function
\begin{eqnarray}\label{2.16}
y(x,\lambda)&=&\rho_{12}\int\limits_{-\pi}^{0} G(x,s;\lambda)f(s)ds+
\rho_{34} \int\limits_{0}^{\pi} G(x,s;\lambda)f(s)ds
\end{eqnarray}
called a resolvent is a solution of the equation
\begin{eqnarray}\label{gr}
y''+\{\lambda-q(x)\}y=f(x), \
\end{eqnarray}
(where $f(x)\neq0$ is a continuous function), satisfying the
boundary-transmission conditions (\ref{1.2})-(\ref{1.5}). We can
assume that $\lambda=0$ is not an eigenvalue. Otherwise, we take a
fixed number $\eta$, and consider the boundary-value-transmission
problem
\begin{equation}\label{.1}
y^{\prime \prime }(x,\lambda)+ \{(\lambda+\eta)-q(x)\}
y(x,\lambda)=0
\end{equation}
\begin{equation}\label{.2}
 a_{1}y'(0-,\lambda)+a_{2}y(0-,\lambda)+a_{3}y'(0+,\lambda)+a_{4}y(0+,\lambda)=0,
\end{equation}
\begin{equation}\label{.3}
 b_{1}y'(0-,\lambda)+b_{2}y(0-,\lambda)+b_{3}y'(0+,\lambda)+b_{4}y(0+,\lambda)=0,
\end{equation}
\begin{equation}\label{14}
 \cos \alpha y(-\pi,\lambda)+\sin\alpha y'(-\pi,\lambda)=0,
\end{equation}
\begin{equation}\label{14}
 \cos\beta y(\pi,\lambda)+\sin\beta y'(\pi,\lambda)=0
\end{equation}
 with the same eigenfunction as for the problem
 (\ref{1.1})-(\ref{1.4}). All the eigenvalues are shifted through
 $\eta$ to the right. It is evident that $\eta$ can be selected so
 that 0 is not an eigenvalue of the new problem.

 Let $G(x,s;0)=G(x,s)$ then the function

 \begin{eqnarray}\label{gr3}
y(x,\lambda)&=& \rho_{12}\int\limits_{-\pi}^{0} G(x,s)f(s)ds +
\rho_{34} \int\limits_{0}^{\pi} G(x,s)f(s)ds
\end{eqnarray}
 is a solution of the equation $y''-q(x)y=f(x),$ and satisfies the boundary-transmission conditions (\ref{1.2})-(\ref{1.4}). We rewrite (\ref{gr}) in the form
\begin{eqnarray}\label{gr5}
y''-q(x)y=f(x)-\lambda y
\end{eqnarray}
Thus, the  homogeneous problem ($f(x)\equiv0$) is equivalent to the
integral equation
\begin{eqnarray}\label{gr7}
y(x,\lambda)+\lambda \{ \rho_{12} \int\limits_{-\pi}^{0}
G(x,s)y(s)ds+ \rho_{34}  \int\limits_{0}^{\pi} G(x,s)y(s)ds \}  =0
\end{eqnarray}

Denote by $\lambda_{0}, \lambda_{1},
\lambda_{2},...,\lambda_{n},...$ the collection of all the
eigenvalues of the problem (\ref{1.1})-(\ref{1.4}), and the
corresponding normalized eigenfunctions by $\varphi_{0},
\varphi_{1}, \varphi_{2},...,\varphi_{n},...$
$$K(x,\xi)=\sum_{n=0}^{\infty}\frac{\varphi_{n}(x)\varphi_{n}(\xi)}{\lambda_{n}}.$$

By the asymptotic formulas for the eigenvalues, obtained in previous
section, the series for $H(x,\xi)$ converges absolutely and
uniformly; therefore, $K(x,\xi)$ is continuous. Consider the kernel
$$P(x,\xi)=G(x,\xi)+ K(x,\xi)=G(x,\xi)+\sum_{n=0}^{\infty}\frac{\varphi_{n}(x)\varphi_{n}(\xi)}{\lambda_{n}}$$
which is obviously continuous and symmetric. By familiar theorem in
the theory of integral equations, any symmetric kernel $P(x,\xi)$
which is not identically zero has at least one eigenfunction
\cite{pe}, i.e., there is a number $\lambda_{0}$ and a function
$u(x)\neq0$ satisfying the equation
\begin{eqnarray}\label{gr8}
u(x,\lambda)+\lambda_{0}\{ \rho_{12} \int\limits_{-\pi}^{0}
P(x,\xi)u(\xi)d\xi+ \rho_{34}  \int\limits_{0}^{\pi}
P(x,\xi)u(\xi)d\xi\}=0 .\end{eqnarray}
 Thus, if we show that the kernel has no
eigenfunctions, we obtain $P(x,\xi)\equiv 0,$ i.e.,
\begin{eqnarray}\label{gr9}
G(x,\xi)=-\sum_{n=0}^{\infty}\frac{\varphi_{n}(x)\varphi_{n}(\xi)}{\lambda_{n}}.
\end{eqnarray}
Hence, to obtain the completeness of eigenfunctions is now easy. It
follows from the equation (\ref{gr7}) that
\begin{eqnarray}\label{gr10}
\rho_{12} \int\limits_{-\pi}^{0} G(x,\xi)\varphi(\xi)d\xi+ \rho_{34}
\int\limits_{0}^{\pi}
G(x,\xi)\varphi(\xi)d\xi=-\frac{1}{\lambda_{n}}\varphi_{n}(x)\end{eqnarray}
therefore,
\begin{eqnarray}\label{gr11}
\rho_{12} \int\limits_{-\pi}^{0} P(x,\xi)\varphi(\xi)d\xi+ \rho_{34}
\int\limits_{0}^{\pi} P(x,\xi)\varphi(\xi)d\xi=0\end{eqnarray} i.e.,
the kernel $Q(x,\xi)$ is orthogonal to all eigenfunctions of the
boundary-value-transmission problem (\ref{1.1})-(\ref{1.4}). Let
$u(x)$ be a solution of the integral equation (\ref{gr8}). We show
that $u(x)$ is orthogonal to all $\varphi_{n}(x)$. In fact, it
follows from (\ref{gr8})
\begin{eqnarray*}\label{gr12} 0= \rho_{12}\int\limits_{-\pi}^{0}u(x)\varphi_{n}(x)dx+
\rho_{34}  \int\limits_{0}^{\pi} u(x)\varphi_{n}(x)\end{eqnarray*}
therefore,
\begin{eqnarray*}\label{gr13} 0&=& u(x,\lambda)+\lambda_{0}\{ \rho_{12} \int\limits_{-\pi}^{0} P(x,\xi)u(\xi)d\xi+
\rho_{34}  \int\limits_{0}^{\pi}
P(x,\xi)u(\xi)d\xi\}\nonumber\\&=&u(x,\lambda)+\lambda_{0}\{
\rho_{12} \int\limits_{-\pi}^{0} G(x,\xi)u(\xi)d\xi+ \rho_{34}
\int\limits_{0}^{\pi} G(x,\xi)u(\xi)d\xi\},\end{eqnarray*} i.e.,
$u(x,\lambda)$ is an eigenfunction of the
boundary-value-transmission problem (\ref{1.1})-(\ref{1.4}). Since
it is orthogonal to all $\varphi_{n}(x)$, it is also orthogonal to
itself, with the consequence that $u(x,\lambda)=0 \ \textrm{and} \
P(x,\xi)=0.$ The formula (\ref{gr9}) is thus proved.
\begin{theorem}\label{teo1}(Expansion Theorem) If $f(x)$ has a
continuous  second derivative and satisfies the
boundary-transmission conditions (\ref{1.2})-(\ref{1.4}), then
$f(x)$ can be expanded into an absolutely and uniformly convergent
Fourier series of eigenfunctions of the  boundary-value-transmission
problem (\ref{1.1})-(\ref{1.4}) on $[-1, 0)\cup(0, 1]$, i.e.,
\begin{eqnarray}\label{gr14}
f(x)=\sum_{n=0}^{\infty}c_{n}\varphi_{n}(x)
\end{eqnarray}
where $c_{n}$ are the Fourier coefficients of f given by
\begin{eqnarray}\label{grr14}
c_{n}=\rho_{12}\int\limits_{-\pi}^{0}f(x)\varphi_{n}(x)dx+ \rho_{34}
\int\limits_{0}^{\pi} f(x)\varphi_{n}(x)dx.
\end{eqnarray}
\end{theorem}
\begin{proof}
\end{proof}
\begin{theorem}(Parseval equality) \label{teo2}For any square-integrable function $f(x)$
in the interval $[-1, 0)\cup (0, 1]$, the Parseval equality
\begin{eqnarray}\label{gr17}
\rho_{12}\int\limits_{-\pi}^{0}f^{2}(x)dx+ \rho_{34}
\int\limits_{0}^{\pi} f^{2}(x)dx=\sum_{n=0}^{\infty}c^{2}_{n}
\end{eqnarray} holds.
\end{theorem}
\begin{proof}
\end{proof}
\section{The modified Carleman  equation%
} We now return to the formula
\begin{eqnarray}\label{gr28}
y(x,\lambda)&=& \rho_{12} \int\limits_{-\pi}^{0}
G(x,s;\lambda)f(s)ds+ \rho_{34}  \int\limits_{0}^{\pi}
G(x,s;\lambda)f(s)ds
\end{eqnarray}
whose right-hand side has been called the resolvent. Let
\begin{eqnarray}\label{gr30}
y(x,\lambda)= \sum_{n=0}^{\infty}b_{n}(\lambda)\varphi_{n}(x), \
c_{n}=\rho_{12}\int\limits_{-\pi}^{0}f(x)\varphi_{n}(x)dx+ \rho_{34}
\int\limits_{0}^{\pi} f(x)\varphi_{n}(x)dx.
\end{eqnarray}
Then, we have
\begin{eqnarray}\label{gr31}
c_{n}&&=\rho_{12}\int\limits_{-\pi}^{0}\{y(x)''-q(x)y(x\}\varphi_{n}(x)dx+
\rho_{34}  \int\limits_{0}^{\pi}
\{y(x)''-q(x)y(x)\}\varphi_{n}(x)dx\nonumber\\&&=-\lambda_{n}b_{n}(\lambda)+b_{n}(\lambda).
\end{eqnarray}
Hence, $b_{n}(\lambda)=\frac{c_{n}}{\lambda-\lambda_{n}}$ and the
expansion of the resolvent is
\begin{eqnarray}\label{gr32}
y(x,\lambda)=\rho_{12} \int\limits_{-\pi}^{0} G(x,s;\lambda)f(s)ds+
\rho_{34} \int\limits_{0}^{\pi} G(x,s;\lambda)f(s)ds=
\sum_{n=0}^{\infty}\frac{c_{n}}{\lambda-\lambda_{n}}
\end{eqnarray}
An important formula can now be derived from the above. Substituting
the equality
\begin{eqnarray}\label{gr33}
c_{n}=\rho_{12}\int\limits_{-\pi}^{0}f(s)\varphi_{n}(s)ds+ \rho_{34}
\int\limits_{0}^{\pi} f(s)\varphi_{n}(s)ds
\end{eqnarray}
on the right-hand side, we see that
\begin{eqnarray}\label{gr34}
&&\rho_{12}\int\limits_{-\pi}^{0} G(x,s;\lambda)f(s)ds+ \rho_{34}
\int\limits_{0}^{\pi}
G(x,s;\lambda)f(s)ds\nonumber\\&&=\sum_{n=0}^{\infty}\frac{\varphi_{n}(x)}{\lambda-\lambda_{n}}\{\rho_{12}\int\limits_{-\pi}^{0}f(s)\varphi_{n}(s)dx+
\rho_{34}  \int\limits_{0}^{\pi} f(s)\varphi_{n}(s)dx\}.
\end{eqnarray}
Since f(s) is arbitrary,
\begin{eqnarray}\label{gr35}
G(x,s;t)=\sum_{n=0}^{\infty}\frac{\varphi_{n}(x)\varphi_{n}(s)}{t-\lambda_{n}}.
\end{eqnarray}
Thus we obtain
\begin{eqnarray}\label{gr36}
 \rho_{12}\int\limits_{-\pi}^{0} G(x,x;t)dt+
\rho_{34}  \int\limits_{0}^{\pi} G(x,x;t)dx=
\sum\limits_{n=0}^{\infty}\frac{1}{t-\lambda_{n}}
\end{eqnarray}
Put $N(\lambda)=\sum_{0\leq\lambda_{n}\leq\lambda}1$ is number of
eigenvalues $\lambda_{n}$ less than $\lambda$. we get from t
(\ref{gr36}) the modified Carleman equation
\begin{eqnarray}\label{gr37}
\rho_{12} \int\limits_{-\pi}^{0} G(x,x;t)dx+ \rho_{34}
\int\limits_{0}^{\pi} G(x,x;t)dx= \sum_{n=0}^{\infty}\frac{d
N(\lambda)}{t-\lambda}.
\end{eqnarray}

\end{document}